\documentclass[12pt]{article}

\usepackage{amssymb}
\usepackage[mathscr]{eucal}

\usepackage{amsmath, amsthm}

\usepackage{amsfonts}

\usepackage{longtable}

\input{xy}
\xyoption{matrix}
\xyoption{arrow}

\newtheorem{theorem}{Theorem}[section]
\newtheorem{lemma}[theorem]{Lemma}

\newtheorem{corollary}[theorem]{Corollary}

\newtheorem{definition}[theorem]{Definition}

\newenvironment{proof of claim}
{\begin{trivlist}  \item \textsc{Proof of Claim:}~} {\hfill $\Box$
\end{trivlist}}

\newcommand{\closure}[1]{\ensuremath{\mathrm{cl}}(#1)}

\newcommand{\Sp}{\mathbb{S}}

\def \st {\operatorname{st}}

\def \d {\operatorname{\delta}}
\def \R{\mathbb{R}}
\def \Q{\mathbb{Q}}
\def \N{\mathbb{N}}
\def \m{\mathfrak{m}}
\def \k{\boldsymbol{k}}
\def \inc{\operatorname{i}}

\def \C{\mathcal{C}}
\def \o{\operatorname{o}}
\def \<{\langle}
\def \>{\rangle}

\def \ind{\operatorname{ind}}

\newcommand{\ma}{\mathfrak{m}}

\def \Def{\operatorname{Def}}
\def \cl {\mathrm{cl}}

\def \cP{\mathscr{P}}
\newcommand{\bk}{\boldsymbol{k}}

\begin{document}
\author{ Lou van den Dries\\Dept. of mathematics, UIUC\\vddries@math.uiuc.edu
\and Jana Ma\v{r}\'{i}kov\'{a}\\Dept. of math. and stat., McMaster
University\\ marikova@math.mcmaster.ca}
\title{Triangulation in o-minimal fields with standard part map}
\maketitle
\begin{abstract}{In answering questions from \cite{st1}
we prove a triangulation result that is of independent interest. In more detail, let $R$ be an
o-minimal field  with a proper convex subring $V$, and let $\st: V \to \bk$
be the corresponding standard part map. Under a mild assumption
on $(R,V)$ we show that definable sets
$X\subseteq V^n$ admit a triangulation that induces a triangulation of its standard
part $\st(X)\subseteq \bk^n$.}
\end{abstract}

\begin{section}{Introduction}

This paper is a sequel to \cite{st1}, and answers some questions it
raised. To discuss this we adopt the notations and conventions from
\cite{st1}. In particular, $R$ is an {\em o-minimal field}, that is, an
o-minimal expansion of an ordered field, $V$ is a proper convex
subring of $R$ with maximal ideal $\m$, ordered residue field
$\bk=V/\ma$, and residue map (or standard part map) $\st \colon V
\to \bk$. For each $n$ this induces $\st \colon V^n \to \bk^n$, and
for $X\subseteq R^n$ we set $\st(X)\colon = \st(X\cap V^n) \subseteq
\bk^n$. The primitives of the expansion $\bk_{\ind}$ of the ordered
field $\bk$ are the ordered ring primitives plus the $n$-ary
relations $\st(X)$ with $X\in \Def^n(R)$, for all $n$. Throughout,
$k,l,m,n$ range over $\N=\{0,1,2,\dots\}$. The problem studied in
\cite{st1} is the following:

\medskip\noindent
{\em What conditions on $(R,V)$ guarantee that $\bk_{\ind}$ is
o-minimal, and what are the definable relations of $\bk_{\ind}$ in that case?}

\medskip\noindent
Here is the main result of \cite{st1} on this issue: {\em
If $(R,V)\models \Sigma_{\inc}$, then for all $n$ the boolean algebra
$\Def^n(\bk_{\ind})$ is generated by $\{\st(X):\ X\in \Def^n(R)\}$.}
By taking $n=1$ it follows that if $(R,V)\models \Sigma_{\inc}$, then
$\bk_{\ind}$ is o-minimal.
Here $\Sigma_{\inc}$ is
a certain first-order axiom scheme to be stated below. It is satisfied in most
cases where $\bk_{\ind}$ was known to be o-minimal: when $V$ is
$\text{Th}(R)$-convex
in the sense of \cite{tconv},
and also when $R$ is
$\omega$-saturated and $V$ is the convex hull of $\Q$ in $R$. (In the latter
case, $\bk$ is isomorphic to the real field $\R$.)

We do not know whether conversely the o-minimality of
$\bk_{\ind}$ implies that $(R,V)\models \Sigma_{\inc}$, but we do prove here the
following converse:

\begin{theorem}\label{converse} If
$\Def^2(\bk_{\ind})$ is generated as a boolean algebra by its subset
$\{\st(X):\ X\in \Def^2(R)\}$, then $(R,V)\models \Sigma$.
\end{theorem}

\noindent
Here $\Sigma$ is a strong version of $\Sigma_{\inc}$.
To define $\Sigma_{\inc}$ and $\Sigma$, put
$$I:= \{x\in R:\ |x| \le 1\},$$ and for
$X\subseteq R^{1+n}$ and $r\in R$, put $X(r):=\{ x\in R^n :\; (r,x)\in X \}$.
``Definable'' means ``definable in $R$ with parameters from $R$'' unless
we specify otherwise. The conditions $\Sigma_{\inc}$ and $\Sigma$ on $(R,V)$ are as follows: \begin{enumerate}
\item[$\Sigma_{\inc}(n)$]: for all definable $X\subseteq I^{1+n}$, if
$X(r)\subseteq X(s)$ for all $r \le s$ in $I$, then
there exists $\epsilon_0 \in \ma^{>0}$ such that $\st X(\epsilon_0 )
=\st X(\epsilon )$ for all $\epsilon\in \ma^{>\epsilon_0 }$;
\item[$\Sigma(n)$]: for all definable $X\subseteq I^{1+n}$
there exists an $\epsilon_0 \in \ma^{>0}$ such that $\st X(\epsilon_0 )
=\st X(\epsilon )$ for all $\epsilon\in \ma^{>\epsilon_0 }$;
\end{enumerate}

\noindent
Also, let $\C(n)$ be the condition that the closed subsets of $\bk^n$
definable in $\bk_{\ind}$ are exactly the sets $\st(X)$ with $X\in \Def^n(R)$.
Finally, $\Sigma_{\inc}$, $\Sigma$, and $\C$ mean
``$\Sigma_{\inc}(n)$ for all $n$'', ``$\Sigma(n)$ for all $n$'', and
``$\C(n)$ for all $n$'', respectively.
Here is a sharper version of Theorem~\ref{converse}, incorporating also
results from \cite{st1}:
\newpage
\begin{theorem}\label{equi} The following conditions on $(R,V)$ are equivalent:
\begin{enumerate}
\item[$(1)$]\  $\C(2)$;
\item[$(2)$]\  $\C$;
\item[$(3)$]\  $\Sigma$;
\item[$(4)$]\  $\Sigma_{\inc}$;
\item[$(5)$]\  $\Def^n(\bk_{\ind})$ is generated by
$\{\st(X):\ X\in \Def^n(R)\}$, for all $n$;
\item[$(6)$]\  $\Def^2(\bk_{\ind})$ is generated by $\{\st(X):\  X\in \Def^2(R)\}$.
\end{enumerate}
\end{theorem}

\noindent
In Section~\ref{cnsn} we prove $(1) \Rightarrow (3)$.
Since $(3) \Rightarrow (2)$
is in \cite{st1} and $(2) \Rightarrow (1)$ is obvious, this yields
the equivalence of conditions (1)--(3). The implications
$(3) \Rightarrow (4)$ and $(5) \Rightarrow (6)$ are also obvious, and
$(4) \Rightarrow (5)$ is in \cite{st1}, but
$(6) \Rightarrow (1)$ requires a new tool: $V$-triangulation.
In Sections 3, 4, 5 we prepare this tool. The main result about it is
Theorem~\ref{tr}; we need only a special case of it
to derive $(6) \Rightarrow (1)$. In triangulating we try to follow
Chapter 8 of \cite{book}, but we have to respect the standard part map and this
requires a lot of extra care. The last Section 7 contains
two more applications of $V$-triangulation.

As to the question whether o-minimality of $\bk_{\ind}$ implies
$\Sigma$, it may be worth considering the case that $V$ is the convex hull of
$\Q$ in $R$ and
the archimedean ordered field $\bk$, upon (unique) identification with an
ordered subfield of the real field $\R$, is all of $\R$.
Then it follows from Baisalov and Poizat~\cite{bp} that $\bk_{\ind}$
is o-minimal, but we do not know if $(R,V)$ necessarily
satisfies the conditions of Theorem~\ref{equi} without making
an extra assumption like ``$R$ is $\omega$-saturated"
or ``$\text{Th}(R)$ has an archimedean model''. There are such cases where these extra assumptions are not satisfied; see \cite{lr} and \cite{uk}.

\bigskip\noindent
{\bf Triangulation respecting the standard part map.}
Our $V$-triangulation result seems of independent interest, and may
be new even when $R$ is a real closed field without
further structure. In
the rest of this Introduction we state it precisely, and define
some notation used throughout the paper.
We let $r,s,t $ (sometimes with subscripts or accents) range over $R$.
For points $a_0,\dots, a_m\in R^n$ (allowing repetitions) we let
$\langle a_0,\dots, a_m\rangle$ denote the affine
span of $\{a_0,\dots, a_m\}$ in $R^n$:
$$\langle a_0,\dots, a_m\rangle =
\{t_0 a_0 + \dots + t_ma_m:\  t_0 + \dots + t_m=1\},$$
and let $[a_0,,\dots, a_m]$ be the convex hull of $\{a_0,\dots, a_m\}$ in $R^n$:
$$[a_0,\dots, a_m] =
\{t_0 a_0 + \dots + t_ma_m:\  t_0 + \dots + t_m=1, \text{ all }t_i\ge 0\}.$$
A {\em simplex in $R^n$\/} is a set $[a_0,\dots, a_m]$
with affinely independent $a_0,\dots, a_m$ in $R^n$, and given such a simplex
$S=[a_0,\dots, a_m]$ we put
$$S^{\o}= (a_0,\dots, a_m):= \{t_0 a_0 + \dots + t_ma_m:\  t_0 + \dots + t_m=1, \text{ all }t_i > 0\},$$
so $S^{\o}$ is the interior of $S$ in its affine span
$\langle a_0,\dots, a_m\rangle$, and $S$ is the closure of $S^{\o}$.
\footnote{Our terminology here is a little different from that in \cite{book}:
there the simplexes were the sets $S^{\o}$, but for the present purpose
it is more convenient for our simplexes to be closed. Likewise, our
definition of ``complex'' and ``triangulation'' here is not exactly
the same as that in \cite{book}, but it is easy to go from one setting
to the other.}
Let $a_0,\dots, a_m\in R^n$ be affinely independent. Then we call
$S=[a_0,\dots, a_m]$ an $m$-simplex. The points $a_0,\dots, a_m$ can
be recovered from $S$ because they
are exactly the extreme points of $S$, as defined in \cite{book}, p.120; they
are also referred to as the {\em vertices\/} of $S$.
A {\em face\/} of $S$ is a simplex $[a_{i_0},\dots, a_{i_k}]$ with
$0\le i_0 < \dots < i_k \le m$.
A {\em complex\/} in $R^n$ is a finite collection $K$
of simplexes in $R^n$ such that each face of each $S\in K$ is in $K$, and for all $S,S'\in K$, if $S\cap S'\ne\emptyset$, then $S\cap S'$ is a common face of
$S$ and $S'$.
For example, the collection of faces of a simplex $S$ in $R^n$ is a
complex $K(S)$ in $R^n$.
Let $K$ be a complex in $R^n$. Then $S^{\o}\cap S'^{\o}=\emptyset$ for all
distinct $S,S'\in K$. Let $|K|$ be the union of the
simplexes in $K$. Then $K^{\o}:=\{S^{\o}: S\in K\}$ is a
finite partition of $|K|$.
A {\em triangulation\/} of a definable $X\subseteq R^n$ is a pair
$(\phi,K)$ consisting of a complex $K$ in $R^n$ and a definable
homeomorphism $\phi \colon X \to |K|$; note that then $X$ is closed
and bounded in $R^n$. Such a triangulation is said to be compatible
with the set $X'\subseteq X$ if $\phi(X')$ is a union of sets
$S^{\o}$ with $S\in K$.

Up to this point this subsection does not require the presence of
$V$ and makes sense for any (not necessarily o-minimal) expansion of an
ordered field in place of $R$, for example $\bk_{\ind}$.

A set $X\subseteq R^n$ is {\em $V$-bounded\/} if for some $r\in V^{>0}$
we have $|x|\le r$ for all $x\in X$. Note that if
$a_0,\dots, a_m\in V^n$, then $[a_0,\dots, a_m]$ is $V$-bounded and
$$\st[a_0,\dots, a_m]= [\st(a_0),\dots, \st(a_m)] \subseteq \bk^n,$$
but if $S$ is a $V$-bounded simplex in $R^n$, then $\st(S)$ is not necessarily
a simplex in $\bk^n$, and even if it is, it might be just a single point
while $S$ is not.

A complex $K$ in $R^n$ is
said to be $V$-bounded if $|K|$ is $V$-bounded.
For a $V$-bounded complex $K$ in $R^n$ we set
$\st(K):=\{\st(S):\ S\in K\}$; this is not always a complex in $\bk^n$,
and even if it is we can have $\st(S)=\st(S')$ with distinct $S,S'\in K$.

A map $f\colon  X \to R^n$ with $X\subseteq R^m$ is said to be
$V$-bounded if its graph $\Gamma(f)\subseteq R^{m+n}$ is
$V$-bounded. Suppose $f\colon X \to R^n$ is definable and
$V$-bounded (so $X$ is definable and $V$-bounded). Then we say that
$f$ {\em induces\/}\footnote{This notion is a little different from
that with the same name in \cite{st1}.} the map $g\colon \st(X) \to
\bk^n$ if $\st(f(x))=g(\st(x))$ for all $x\in X$, equivalently,
$\st(\Gamma(f))=\Gamma(g)$; note that then $g$ is definable in
$\bk_{\ind}$.

\medskip\noindent
A {\em $V$-triangulation\/} of a definable $V$-bounded $X\subseteq
R^n$ is a triangulation $(\phi,K)$ of $X$ such that $K$ is
$V$-bounded, $\phi$ induces a map $\phi_{\st}\colon \st X \to
\st(|K|)$, and $(\phi_{\st}, \st(K))$ is a triangulation of $\st(X)$
in the sense of $\bk_{\ind}$.

\medskip\noindent
With this terminology in place we can state our triangulation theorem:

\begin{theorem}\label{triangthm} If $\Def^2(\bk_{\ind})$ is generated by
$\{\st(X):\ X\in \Def^2(R)\}$, then
each definable closed $V$-bounded set
$X\subseteq R^{n}$ with definable $X_1 , \dots , X_k\subseteq X$ has a
$V$-triangulation compatible with $X_1 , \dots, X_k$.
\end{theorem}

\noindent
We finish this introduction with some more notation and a useful fact.
Let $n\ge 1$. For $x=(x_1,\dots, x_n)\in R^n$ and $y=(y_1,\dots, y_n)\in R^n$
we set
$$d(x,y):= \max\{|x_i-y_i|:\ i=1,\dots,n\}.$$ Likewise for $x,y\in \bk^n$.

\begin{lemma} Suppose $X\subseteq R^n$ and $f\colon X \to R$ are definable and
$V$-bounded, and $f$ induces a function $g\colon \st X \to \bk$.
Then $g$ is continuous.
\end{lemma}
\begin{proof} We can assume $n\ge 1$. Assume towards a contradiction that
$a\in X$ and $\epsilon\in R^{>\ma}$
are such that for every $\delta \in R^{>\ma}$ there is
$x \in X$ such that $d(\st a ,\st x )<
\st \delta$ and $|g(\st a )- g(\st x)| >\st
\epsilon$. Then the set
$$\{ r \in R^{>0}:\; \text{there is }x \in X \text{ such that } d(a,x)< r
\mbox{ and } |f(a)-f(x)|> \epsilon )\}$$ has an element in $\ma$, by o-minimality of $R$. This contradicts that $f$ induces a function.
\end{proof}

\end{section}

\begin{section}{$\C(2) \Rightarrow \Sigma$}\label{cnsn}

\noindent
The conditions $\C(n)$ and $\Sigma(n)$ were defined in the introduction.
We begin with some observations about the case $n=1$. It is clear that
the o-minimality of  $\bk_{\ind}$ is equivalent to the condition that
$\Def^1(\bk_{\ind})$ is generated as a boolean algebra
by $\{\st(X):\ X\in \Def^1(R)\}$.
Next, we have the equivalence
$$ \boldsymbol{k}_{\ind} \text{ is o-minimal}\  \Longleftrightarrow\ \C(1).$$
The forward direction is obvious. For the converse, use that
$\boldsymbol{k}_{\ind}$ is weakly  o-minimal, by \cite{bp}, and
that if $Y\subseteq \bk$ is bounded and convex in $\bk$, and has neither
infimum nor supremum in $\bk$, then $Y$ is closed (and open) in $\bk$.

It follows easily by cell decomposition that $\Sigma (1)$ is equivalent
to the condition that for all definable $f \colon I\rightarrow I$
there is $\epsilon_0 \in \ma^{>0}$ such that
$\st f(\epsilon )=\st f(\epsilon_0 )$ for all $\epsilon \in \ma^{>\epsilon_0 }$.
Later in this section we prove that for all $n$ we have
$\C(n)\Rightarrow \Sigma(n)$, and so by the above this gives
$$ \bk_{\ind} \text{ is o-minimal}\  \Longrightarrow\ \Sigma(1),$$
but we do not know if the converse holds.
We have $\C(n) \Rightarrow \C(m)$ for $n>m$ since projection maps and
standard part maps commute. In particular, if $\C(n)$ holds for some
$n\ge 1$, then $\bk_{\ind}$  is o-minimal.

\medskip\noindent
It is convenient to introduce the following weak version $\C'(2)$ of $\C(2)$:

\smallskip
{\em for all continuous $\phi \colon I(\bk) \to \bk$ that are
definable in $\k_{\ind}$ there exists a

set $X\in \Def^2(R)$ such that $\Gamma(\phi)=\st(X)$}.

\smallskip\noindent
Using as above the weak o-minimality of $\bk_{\ind}$ we see that
$\C'(2)\Rightarrow \C(1)$, and in particular
$\C'(2)\Rightarrow  \bk_{\ind} \text{ is o-minimal}$. {\bf In the rest of this
section we assume that $\bk_{\ind}$  is o-minimal}.

\begin{lemma}\label{c2a} Let $\phi \colon I(\bk) \to \bk$ be continuous and definable in $\bk_{\ind}$, and
suppose $\Gamma(\phi)=\st(X)$ for some $X\in \Def^2(R)$. Then $\phi$
is induced by some $V$-bounded continuous definable $f\colon I \to
R$.
\end{lemma}
\begin{proof} Take a $V$-bounded closed $X\subseteq I\times R$ with
$X\in \Def^2(R)$ such that $\Gamma(\phi)=\st(X)$. Let $p\colon R^2
\to R$ be the projection map given by $p(x,y)=x$. Then
$p(X)\subseteq I$ with $\st(p(X))=I(\bk)$, and by definable choice
we have a definable $h\colon p(X)\to R$ such that
$\Gamma(h)\subseteq X$. By the piecewise continuity of $h$ we can
shrink $p(X)$ slightly to a closed definable set $P\subseteq p(X)$
such that $\st(P)=I(\bk)$ and $g:= h|P$ is continuous. Since
$\st(\Gamma(g))\subseteq \Gamma(\phi)$, it follows that
$\st(\Gamma(g))=\Gamma(\phi)$. In particular, for $a,b\in P$ with $a
< b$ and $(a,b)\cap P=\emptyset$, we have $\st(a)=\st(b)$ and
$\st(g(a))=\st(g(b))$. It is easy to extend $g$ to a continuous
definable $f\colon I \to R$ such that for all $a,b\in P$ as before
$f$ is monotone on $[a,b]$. It follows that $f$ is $V$-bounded and
$\st(\Gamma(f))=\Gamma(\phi)$, so $f$ induces $\phi$.
\end{proof}

\begin{lemma}\label{c2b} Suppose $\C'(2)$ holds. Then each closed bounded definable $Y\in \Def^n(\bk_{\ind})$ with $\dim Y \le 1$ equals $\st(X)$ for some
$X\in \Def^n(R)$.
\end{lemma}
\begin{proof} This is clear for $n=1$, so let $Y\in \Def^n(\bk_{\ind})$ be
closed and bounded with $\dim Y \le 1$, $n>1$. For a permutation $\sigma$
of $\{1,\dots,n\}$ and $Z\subseteq \bk^n$ we put
$$\sigma(Z):= \{(a_{\sigma(1)},\dots, a_{\sigma(n)}):\ (a_1,\dots, a_n)\in Z\}.$$
The assumptions on $Y$ yield that $Y$ is a finite union of sets
$\sigma(\Gamma(\phi))$ where $\phi \colon [a,b] \to \bk^{n-1}$ is
continuous and definable in $\bk_{\ind}$, $a \le b$ in $\bk$, and
$\sigma$ is a permutation of $\{1,\dots,n\}$. So in order to show
that $Y=\st(X)$ for some definable $X\in \Def^n(R)$ we can assume
that $Y=\Gamma(\phi)$ where $\phi \colon [a,b] \to \bk^{n-1}$ is
continuous and definable in $\bk_{\ind}$ and $a \le b$ in $\bk$. If
$a=b$, then $Y$ is a singleton, and there is no problem, so we can
assume $a < b$. Then we can reduce to the case that $a=0, b=1$, so
$\phi= (\phi_1,\dots, \phi_{n-1}) \colon I(\bk) \to \bk^{n-1}$.  By
the previous lemma and the hypothesis of the present lemma we have
$V$-bounded definable continuous $f_1,\dots, f_{n-1}\colon I \to R$
that induce $\phi_1,\dots, \phi_{n-1}$, so $f=(f_1,\dots,
f_{n-1})\colon I \to R^{n-1}$ induces $\phi$, and thus
$\st(\Gamma(f))=\Gamma(\phi)=Y$, as desired.
\end{proof}

\noindent
The next two definitions and Lemma~\ref{hl} use only that $R$ is an o-minimal field
and do not require a proper convex subring $V$ of $R$, so they apply also to
$\bk_{\ind}$ (which we have assumed to be o-minimal).
For results about hausdorff limits of definable families in an
o-minimal expansion of the real field, see \cite{hlimit}. The notion
of hausdorff distance is also useful in our setting of an arbitrary
o-minimal field.

\begin{definition} Let $n\ge 1$ and put $d(x,Y):= \inf\{d(x,y): y\in Y\}$ for $x\in R^n$ and nonempty definable $Y\subseteq R^n$. Next, let $X,Y \subseteq R^n$ be definable, closed, bounded and nonempty.
Then the {\em hausdorff distance\/} between $X$ and $Y$ is defined
to be
$$d_H (X,Y):= \min\{r\geq 0: d(x,Y), d(y,X)\leq r
\mbox{ for all }x\in X \mbox{ and all }y\in Y   \}.$$
\end{definition}

\noindent
With $X$,$Y$ as in this definition, note that $d_H(X,Y)\in
R^{\ge 0}$, $d_H(X,Y)=0$ iff $X=Y$, $d_H (X,Y) = d_H (Y,X)$, and
whenever $Z \subseteq R^n$ is also definable, closed, bounded, and
nonempty, then
$$d_H (X,Z)\le d_H (X,Y) + d_H (Y,Z).$$ So $d_H$ behaves like a
metric (but takes values in $R^{\ge 0}$ rather than $\R^{\ge 0}$).

\begin{definition}
Let $n\ge 1$ and let $X\subseteq R^{1+n}$ be definable such that
the set $X(t)\subseteq R^n$ is closed,
bounded and nonempty for all sufficiently small $t>0$.
Then a {\em hausdorff limit\/} of $X(t)$ as $t
\to 0+$ is a definable, closed, bounded, and nonempty set
$H\subseteq R^n$ such that $\lim_{t \to 0+} d_H(X(t), H)=0$.
\end{definition}

\begin{lemma}\label{hl}
Let $n\ge 1$, and let $X\subseteq R^{1+n}$  be definable such that
$X(t)$ is closed, bounded and nonempty for all sufficiently small $t>0$.
Then there is a unique hausdorff limit of $X(t)$ as $t\to 0+$. Moreover,
if $\dim X(t) \le m$ for all $t>0$, then $\dim H \le m$ for this hausdorff limit $H$.
\end{lemma}
\begin{proof}
If $H,H'$ are hausdorff limits of $X(t)$ as
$t\rightarrow 0+$, then $d_H (H,H')=0$, hence $H=H'$. So there is at most one hausdorff
limit of $X(t)$ as $t\rightarrow 0+$.

To show existence, let
$Y:= \{(t,x)\in X:\ t>0\}$ and $H:= \closure{Y}(0)\subseteq R^n$.
It is clear that $H$ is definable, closed, bounded, and, by cell
decomposition, $H$ is nonempty. We claim that $\lim_{t\to 0+}d_H
(H,X(t))=0$. Suppose not. Then we have $\delta
>0$ such that for every $s>0$ there is a positive $t<s$ with $d_H
(H,X(t))
>\delta$. By o-minimality we can take $s>0$ such that
either there is for every $t\in
(0,s)$ a point $a_t \in H$ with $d(a_t , X(t))>\delta$, or
there is for every $t\in (0,s)$ a point $b_t \in X(t)$ with
$d (b_t ,H)>\delta$. In the first case we can assume by definable choice that
$t \mapsto a_t \colon (0,s)\rightarrow H$ is definable; set
$a:= \lim_{t\to 0+} a_t$. Then
$a\in H$ since $H$ is closed, but it is also easy to check that $(0,a)\notin \cl(Y)$,
so $a\notin H$, a contradiction.
In the second case we can assume that $t\mapsto b_t \colon (0,s) \to R^n$ is definable; set $b:= \lim_{t\to 0+} b_t$. Then $(0,b)\in \cl(Y)$, so $b\in H$, but also
$d(b,H) \ge \delta$, a contradiction.

Suppose now that $\dim X(t) \le m$ for all $t>0$. Then for $Y\subseteq X$ and $H=\cl(Y)(0)$ as above we have $\{0\}\times H\subseteq \cl(Y) \setminus Y$, so
$\dim H \le m$.
\end{proof}

\medskip\noindent
In the rest of this section $s, s',s_0, s_1, s_2$ range over $I$.

\begin{lemma}\label{tlemma} Let $s_0\in \ma$ and $s_1>\ma$, and suppose
$f \colon (s_0, s_1) \to R$ is definable and $f(s)\in \ma$ for all $s$ with $\ma < s < s_1$. Then there is a
$p\in \ma^{>s_0}$ such that $f(s)\in \ma$ for all $s$ with $p\le s
<s_1$.
\end{lemma}
\begin{proof} We can assume that $f$ is of class $C^1$. For $\ma < s < s_1$
we have
$$f(s) - f(s/2)= (s/2)f'(s')$$ with $s'\in [s/2,s]$, so $f'(s')\in
\ma$. It follows that the definable set $$\{x\in (s_0,s_1):\
|f'(x)|\le 1\}$$ contains a set $[p,q]$ with $s_0 < p \in \ma$ and
$\ma < q < s_1$. Let $s$ with $p \le s \in \ma$ be fixed, and take a
variable $s'$ with $\ma < s' \le q$. Then $f(s')-f(s) = (s'-s)f'(x)$
with $s\le x \le s'$, so $|f(s')-f(s)|\le s'-s \le s'$. Since
$f(s')\in \ma$ and we can take $s'$ arbitrarily small, subject to
$s'>\ma$, we obtain $f(s)\in \ma$.
\end{proof}

\noindent
More than the result itself, the {\em proof\/} of the following is crucial.

\begin{lemma}\label{closed then sigma} For all $n\ge 1$ we have
$\mathcal{C}(n)\ \Longrightarrow\ \Sigma(n)$.
\end{lemma}
\begin{proof} Let $n\ge 1$, and assume $(R,V)$ satisfies $\mathcal{C}(n)$.
Let $X\subseteq I^{1+n}$ be
definable. Our job is to show the existence of $\epsilon_0 \in \ma^{>0}$
such that $\st X(\epsilon_0 )
=\st X(\epsilon )$ for all $\epsilon\in \ma^{>\epsilon_0 }$. We can assume that
$X(s)\ne \emptyset$ for all $s\in I$. Put $Y:= \st X \subseteq
I(\boldsymbol{k})^{1+n}$. Let $Q\subseteq I(\boldsymbol{k} )^n$ be
the hausdorff limit of $Y(t)$ as $t>0$ tends to $0$ in
$\boldsymbol{k}$, so $Q$ is definable in
$\boldsymbol{k}_{\text{ind}}$, and $Q$ is closed, bounded, and nonempty.

Using $\mathcal{C}(n)$ we can take a closed definable $P\subseteq I^n$ such that $\st P=Q$.

\smallskip\noindent{\em Claim.} Let $\delta\in R^{>\ma}$. Then there
is an $s'>\ma$ such that $d_H(X(s),P) < \delta$ for all $s$ with
$\ma < s < s'$.

\smallskip\noindent
Suppose the claim is false. This gives $s_0,s_1$ with $s_0\in \ma < s_1$
such that \begin{align*} d_H(X(s),P)&\ge \delta\
\text{ for all }s\in (s_0, s_1),\\
 d_H(Y(t), Q) &< \st (\delta)\
 \text{ for all $t$ with }0<t\le t_1:=\st (s_1).
\end{align*}
Let $\ma < s < s_1$. Then $d_H(X(s),P) \ge \delta$ gives either an
$x\in X(s)$ with $d(x,P) \ge \delta$, or a point $p\in P$ with
$d(X(s),p) \ge \delta$. But $x\in X(s)$ with $d(x,P)\ge \delta$
would give $\st (x)\in Y(\st (s) )$ with $0< \st (s) \le t_1$ and
$d(\st (x) ,Q) \ge \st (\delta)$, a contradiction. Thus $d(X(s),p)
\ge \delta$ for some $p\in P$.

By increasing $s_0$ we can therefore arrange that for each $s\in
(s_0,s_1)$ there is $p\in P$ with $d(X(s),p)\ge \delta$. Definable
choice gives a definable function $p\colon (s_0, s_1) \to P$ such
that $d(X(s),p(s)) \ge \delta$ for all $s\in (s_0,s_1)$. Definable
choice in the structure $\boldsymbol{k}_{\text{ind}}$ then gives a
function $q \colon (0, t_1) \to Q$, definable in
$\boldsymbol{k}_{\text{ind}}$, such that $(t, q(t))\in \st
(\Gamma(p))$ for all $t\in (0, t_1)$. Let $q_0:= \lim_{t\to 0}
q(t)$, so $q_0\in Q$. Take $p_0\in P$ with $\st (p_0)=q_0$. Since
$d(q(t), q_0) < \st (\delta)/2$ for all sufficiently small $t>0$,
there is for each $s'$ with $\ma < s' < s_1$ an $s$ such that $\ma <
s < s'$ and $d(p(s), p_0) < \delta/2$, and thus $d(X(s), p_0) >
\delta/2$. Then by the o-minimality of $R$ we can decrease $s_1$ to
arrange that $d(X(s), p_0) > \delta/2$ for all $s$ with $\ma < s <
s_1$ and $d(Y(t), q_0) < \st (\delta)/2$ for all $t$ with $0<t<
t_1$. For such $t$, take $y\in Y(t)$ with $d(y, q_0) < \st
(\delta)/2$; then $(t,y)\in Y$, so $(t,y)= \st (s,x)$ with $(s,x)\in
X$, so $d(x,p_0) < \delta/2$ with $\ma < s < s_1$ and $x\in X(s)$, a
contradiction. This concludes the proof of the claim.

\medskip\noindent
Define $f \colon I \to R$ by $f(s)=d_H(X(s), P)$, so $f$ is
definable. Changing $s_0$ and $s_1$, if need be, we arrange that the
restriction of $f$ to $[s_0, s_1]$ is continuous and monotone. If
this restriction of $f$ is increasing, then it follows from the
Claim that $f(\epsilon)\in \ma$ for all $\epsilon\in \ma$ with
$\epsilon\ge s_0$, and thus $d_H(X(\epsilon), X(s_0))\in \ma$ for
all such $\epsilon$, that is, $\st X(\epsilon)\ =\ \st X(s_0)$ for
all such $\epsilon$, and we are done.

So we can assume for the rest of the proof that $f$ is decreasing on
$(s_0,s_1)$. Then $f(s)\in \ma$ for all $s$ with $\ma < s  < s_1$ by
the Claim. Then Lemma~\ref{tlemma} gives $\epsilon_0\in \ma,\
\epsilon_0>s_0$, such that $f(s)\in \ma$ for all $s$ with
$\epsilon_0\le s <s_1$. As before, this yields $\st X(\epsilon)\ =\
\st X(\epsilon_0)$ for all $\epsilon\ge \epsilon_0$ in $\ma$.
\end{proof}

\noindent
Next a reduction to $1$-parameter families of $1$-dimensional sets:

\begin{lemma}~\label{red11} Let $n\ge 1$, and suppose that for all
definable $X\subseteq I^{1+n}$ with $\dim X(r) \le 1$ for all $r\in I$
there is $\epsilon_0 \in \ma^{>0}$ such that $\st X(\epsilon_0 )
=\st X(\epsilon )$ for all $\epsilon\in \ma^{>\epsilon_0 }$.
Then $(R,V)\models \Sigma(n)$.
\end{lemma}
\begin{proof}
Let $X\subseteq I^{1+n}$ be definable
such that $X(r)$ is nonempty for all $r\in I$. Definable choice
gives a definable map $f \colon I^2 \to I^n$ such that
for all $r,s\in I$ we have $f(r,s)\in X(r)$ and
$$d(f(r,s), X(s)) = \sup\{d(x, X(s)):\ x\in X(r)\}.$$
Define $Y\subseteq I^{1+n}$ by $Y(r)=\{f(r,s): \ s\in I \}$
for $r\in I$. Then $Y(r)\subseteq X(r)$, $\dim Y(r) \le 1$, and
$d_H(Y(r), Y(s)) \ge d_H(X(r), X(s))$, for all $r,s\in I$. The hypothesis
of the lemma gives $\epsilon_0 \in \ma^{>0}$ such that $\st Y(\epsilon_0 )
=\st Y(\epsilon )$ for all $\epsilon\in \ma^{>\epsilon_0 }$, hence
$d_H\big(Y(\epsilon_0 ),Y(\epsilon )\big)\in \ma$ for all
$\epsilon\in \ma^{>\epsilon_0 }$, so $d_H\big(X(\epsilon_0 ),X(\epsilon )\big)\in \ma$ for all
$\epsilon\in \ma^{>\epsilon_0 }$, and thus
$\st X(\epsilon_0 )
=\st X(\epsilon )$ for all $\epsilon\in \ma^{>\epsilon_0 }$.
\end{proof}


\begin{corollary} $\mathcal{C}'(2)\ \Longrightarrow\ \Sigma$.
\end{corollary}
\begin{proof} Assume $(R,V)$ satisfies $\mathcal{C}'(2)$. Towards proving
$(R,V) \models \Sigma$, consider
a definable $X\subseteq I^{1+n}$, $n\ge 1$, with $X(r)\ne \emptyset$ and
$\dim X(r) \le 1$ for all $r\in I$; by Lemma~\ref{red11} it suffices to show
that then there is $\epsilon_0 \in \ma^{>0}$ such that $\st X(\epsilon_0 )
=\st X(\epsilon )$ for all $\epsilon\in \ma^{>\epsilon_0 }$. Towards this goal
we use the proof of Lemma~\ref{closed then sigma}. The present $X$ satisfies
$\dim X \le 2$, so the set
$Y=\st(X)$ in that proof has also dimension at most 2 in the sense of the
o-minimal structure  $\bk_{\ind}$, by Corollary 2.8 in \cite{st1}.
So for all but finitely many $t\in I(\bk)$ the section
$Y(t)$ has dimension
at most 1. As in that proof, let $Q\subseteq I(\bk)^n$ be the hausdorff limit
of $Y(t)$ as $t>0$ tends to $0$ in $\bk$. Then $\dim Q \le 1$, by Lemma~\ref{hl}. In the proof of
 Lemma~\ref{closed then sigma} we only used the assumption $\C(n)$ to provide a
$P\in \Def^n(R)$ with $\st(P)=Q$. Since $\dim Q \le 1$ and
$\C'(2)$ holds, we can appeal here instead to Lemma~\ref{c2b} to provide
such a $P$.
With this $P$ the rest of the proof of  Lemma~\ref{closed then sigma} goes
through to give an $\epsilon_0$ as desired.
\end{proof}

\noindent
In particular, we have $\C(2)\ \Rightarrow\ \Sigma$, and together with
the results from \cite{st1} this gives the equivalence of conditions
(1)--(3) of Theorem~\ref{equi}. Of course, these conditions are also
equivalent to $\C'(2)$.

\end{section}

\begin{section}{Construction of a complex}\label{cc}

\noindent
As we mentioned in the introduction, we shall adapt to our purpose
the proof of the o-minimal triangulation theorem in Chapter 8 of \cite{book}.
The first non-trivial issue that comes up in doing this is of a purely
semilinear nature,
and consists of finding a version of Lemma (1.10) in that chapter that
is compatible with the standard part map. That lemma constructs a complex in
$R^{n+1}$ based on a simplex in $R^n$, and to make this construction
compatible with the standard part map we need to linearly
order the vertices of the simplex in a special way.

\medskip\noindent
In more detail, recall that if $S$ is a simplex
in $R^n$, then $K(S)$ is the complex in $R^n$ whose elements are the faces of $S$.
Define a {\em $V$-simplex in $R^n$\/} to be a $V$-bounded
simplex $S\subseteq R^n$ such that $\st(K(S))$ is a complex in $\bk^n$. Note that
if $S$ is a $V$-simplex in $R^n$, then $\st(S)$ is a simplex in $\bk^n$ and
$K(\st(S))=\st(K(S))$.  We also define a {\em $V$-complex in
$R^n$\/} to be a $V$-bounded complex $K$ in $R^n$ such that $\st K$ is a
complex in $\boldsymbol{k}^n$; note that then the simplexes of $K$ are $V$-simplexes.

Given a $V$-simplex $S$ in $R^n$ our construction will require
an ordering $a_0 < a_1 < \dots < a_m$
of its vertices and indices $0=i_0 <i_1 < \dots < i_k\le m$ such that
$\st(a_{i_0}),\dots, \st(a_{i_k})$ are the distinct vertices of $\st S$, and
$\st(a_{i_\kappa})=\st(a_i)$ whenever $i_{\kappa} \le i < i_{\kappa+1}$.
This suggests the following notion (to be applied to the standard parts of
points in $V^n$).

\medskip\noindent
Let $a_0,\dots, a_m\in \bk^n$.
We call the sequence $a_0,\dots, a_m$ {\em simplicial\/} if there are indices
$i_0 < \dots < i_k$ in $\{0,\dots, m\}$ with $i_0=0$ such that
$a_{i_0},\dots, a_{i_k}$ are affinely independent in $\bk^n$, and
$a_{i_\kappa}=a_i$ whenever
$$0\le \kappa \le k, \quad i_{\kappa} \le i < i_{\kappa+1}, \quad (\text{with }i_{k+1}:= m+1).$$

\medskip\noindent
Suppose the sequence $a_0,\dots, a_m$ is {\em simplicial\/} and let
$i_0,\dots, i_k$ be as above. Then $[a_0,\dots, a_m]=[a_{i_0},\dots, a_{i_k}]$
is a $k$-simplex; if
$0\le j_0 < \dots < j_l\le m$, then the sequence $a_{j_0},\dots, a_{j_l}$ is also
simplicial, and $[a_{j_0},\dots, a_{j_l}]$ is a face of $[a_0,\dots, a_m]$;
all faces of $[a_0,\dots, a_m]$ are obtained in this way, but different
sequences $j_0,\dots, j_l$ can give the same face.

Let $r_i, s_i\in \bk$ for $i=0,\dots,m$ be such that $r_i \le s_i$ for all
$i$ and
\begin{align*} &r_{i_\kappa}=r_i \text{ and }s_{i_\kappa}=s_i
\text{ whenever}\\
&0\le \kappa \le k, \quad i_{\kappa} \le i < i_{\kappa+1}, \quad (\text{with }i_{k+1}:= m+1).
\end{align*}
Put $b_i:= (a_i,r_i),\ c_i=(a_i,s_i)$ (points in $\bk^{n+1}$). Then we have the
following variant of Lemma (1.10) in Chapter 8 of \cite{book}.

\begin{lemma}\label{vs} If
$0\le j_0 < \dots <j_p \le j_{p+1} < \dots < j_q\le m$, $p<q$,
then the sequence $b_{j_0},\dots, b_{j_p}, c_{j_{p+1}},\dots, c_{j_q}$ is
simplicial. Let $L$ be the set of all simplexes
$[b_{j_0},\dots, b_{j_p}, c_{j_{p+1}},\dots, c_{j_q}]$ obtained from
such sequences $j_0,\dots, j_q$, and all faces of these simplexes. Then $L$ is a complex with
\begin{align*} |L|&=\{t(t_0b_0+\dots + t_mb_m)+(1-t)(t_0c_0+ \dots + t_mc_m):\\
  & \qquad \qquad \qquad \qquad \qquad 0\le t \le 1,\ \text{ all }t_i\ge 0,\
                             t_0 + \dots + t_m=1\}\\
&=\ \text{ convex hull of }\{b_0,\dots, b_m, c_0,\dots, c_m\}.
\end{align*}
\end{lemma}
\begin{proof} It is routine to check that the first statement is true. As to
the rest, consider first the case that $r_i=s_i$ for all $i$. Then $b_i=c_i$
for all $i$, so $L$ is just the set of faces of the $k$-simplex
$[b_0,\dots, b_m]$, and the claim about $|L|$ then holds trivially.
Suppose $r_i < s_i$ for some $i$. Then, if $0\le p \le k$ and
$r_{i_p} < s_{i_p}$ we have a $(k+1)$-simplex
$[b_{i_0},\dots, b_{i_p}, c_{i_p}, \dots, c_{i_k}]\in L$. It is routine to
check that $L$ is the
set of the $(k+1)$-simplexes obtained in this way and all their faces.
Then our claim follows from Lemma (1.10) in Chapter 8 of \cite{book}.
\end{proof}

\begin{lemma}\label{vg} Let $S$ be a $V$-bounded $m$-simplex in $R^n$. Then
the following are equivalent: \begin{enumerate}
\item[$(1)$] $S$ is a $V$-simplex;
\item[$(2)$] there is an enumeration $a_0,\dots, a_m$ of the vertices of $S$
such that the sequence $\st(a_0),\dots, \st(a_m)$ is simplicial.
\end{enumerate}
\end{lemma}
\begin{proof} Assume (1). Then every vertex $a$ of $S$ yields a
vertex $\st(a)$ of the simplex $\st(S)$, so with
$k:=\dim\big( \st (S)\big)$
we have an enumeration of the vertices of $S$ and indices
$i_0 <\dots < i_k$ as in (2).

It is routine to check that (2) implies (1).
\end{proof}

\noindent
Let $S$ be a $V$-simplex. Then Lemma~\ref{vg}  yields
an enumeration $a_0,\dots, a_m$ of its vertices and indices
$0=i_0< \dots < i_k$ in $\{0,\dots, m\}$ such that
$\st(a_{i_0}),\dots, \st(a_{i_k})$ are affinely independent in $\bk^n$, and
$\st(a_{i_\kappa})=\st(a_i)$ whenever
$$0\le \kappa \le k, \quad i_{\kappa} \le i < i_{\kappa+1}, \quad (\text{with }i_{k+1}:= m+1).$$
Let $r_i, s_i\in V$ for $i=0,\dots,m$ be such that $r_i \le s_i$ for all $i$,
\begin{align*} &\st(r_i)=\st(r_{i_\kappa}) \text{ and }\st(s_i)=\st(s_{i_\kappa})
\text{ whenever}\\
&i_{\kappa} \le i < i_{\kappa+1}, \quad 0\le \kappa \le k\quad (\text{with }i_{k+1}:= m+1),
\end{align*}
and $r_i < s_i$ for some $i$. Put $b_i:= (a_i,r_i),\ c_i=(a_i,s_i)$ (points in
$V^{n+1}$). Let $L$ be the set of all $(m+1)$-simplexes
$[b_0,\dots, b_i, c_i,\dots, c_m]$ with $b_i\ne c_i$,
and all faces of these simplexes.
Then by Lemma (1.10) of Chapter 8 in \cite{book}, $L$ is a complex with
$$|L|=\text{convex hull of }\{b_0,\dots, b_m, c_0,\dots, c_m\}.$$

\begin{corollary}\label{vgood} $L$ is a $V$-complex.
\end{corollary}
\begin{proof} It follows easily from the assumptions on $r_i, s_i$
that each simplex $[b_0,\dots, b_i, c_i,\dots, c_m]$ with $b_i\ne c_i$
is a $V$-simplex. A face of a $V$-simplex is also a $V$-simplex, so each simplex of
$L$ is a $V$-simplex. That $\st(L)$ is a complex follows from
Lemma~\ref{vs} with the $\st(b_i), \st(c_j)$ in the role of $b_i, c_j$.
\end{proof}

\end{section}

\begin{section}{Extension Lemmas}

\noindent
The first extension lemma below is a $V$-version of lemma (2.1)
in Chapter 8 of \cite{book}, but requires a very different proof.
Before stating it we make some preliminary remarks and definitions.

\medskip\noindent
First, let $E$ be an affine subspace of $R^n$ of dimension $k\ge 1$,
so $E=e+L$ with $e\in R^n$ and $L$ a linear subspace of $R^n$ of dimension $k$.
Let $H_1$ and $H_2$ be affine hyperplanes in $E$, so
$H_i=e_i+L_i$ with $e_i\in E$ and a linear subspace $L_i$ of $L$ of dimension
$k-1$, for $i=1,2$. Let $u\in L\setminus (L_1\cup L_2)$. Then we have a direct
sum decomposition $L=Ru\oplus L_2$, which yields a map
$$(H_1, H_2): H_1 \to H_2, \quad \{(H_1,H_2)(x)\}= (x+Ru)\cap H_2 \text{ for all }x\in H_1.$$
This map is easily seen to be affine, and thus continuous, and to be a
bijection with inverse $(H_2, H_1)$.

Next, let $S$ be a simplex in $R^n$. A {\em proper face\/} of $S$ is a face
$F$ of $S$ such
that $F\ne S$. We set $\d(S):= \text{union of the proper faces of }S$;
this is the topological boundary of $S$ in the affine span of its vertices.

These definitions and remarks go through
for any ordered field instead of $R$, for example $\bk$.
In the rest of this section we assume that $\bk_{\ind}$ is o-minimal.

\begin{lemma}\label{ext1}
Let $S \subseteq R^n$ be a $V$-bounded simplex and let $f\colon
\d(S) \to R$ be a continuous $V$-bounded definable function inducing
a function $\st \d(S)\to \bk$. Then $f$ has a continuous $V$-bounded
definable extension $g\colon S \to R$ inducing a function $\st S \to
\bk$.
\end{lemma}
\begin{proof} Let $a_0,\dots, a_k$ be the distinct vertices of $S$.
The lemma holds trivially for $k=0$
since $\d(S)=\emptyset$ in that case. So let $k\ge 1$, and let
$E=\<a_0,\dots, a_k\>$ be the affine span of the vertices of $S$.
Below $i$ ranges over
$\{0,\dots,k\}$, and we set
$$H_i:= \<a_0,\dots, a_{i-1}, a_{i+1},\dots, a_k\>, \quad
F_i:= [a_0,\dots, a_{i-1}, a_{i+1},\dots, a_k].$$
Let $L$ be the linear subspace of $R^n$ of which $E$ is a translate, and
let $L_i$ be the proper linear subspace of $L$ of which
$H_i$ is a translate. Take a vector $u\in L\setminus \bigcup_i L_i$.
(Later in the proof we impose further restrictions on $u$.)
For $x\in \d(S)$ we have $(x+Ru)\cap S=[x,y]$ for a unique $y\in S$, and for
this $y$
we have $(x+Ru)\cap \d(S)=\{x,y\}$; we define $\lambda: \d(S) \to \d(S)$ by
$\lambda(x)=y$ for $y$ as above. Note that $\lambda\circ \lambda=\lambda$.

\medskip\noindent
{\em Claim 1 \/}: $\lambda$ is continuous. To see this we note that
the closed subsets $F_i \cap (H_i, H_j)^{-1}(F_j),\ (0\le i,j \le k)$
of $\d(S)$ cover $\d(S)$. In view of the first remark
of this section $\lambda$ agrees on each such
$F_i\cap (H_i, H_j)^{-1}(F_j)$ with the continuous map $(H_i, H_j)$.

\medskip\noindent
We now extend $f$ to $g\colon S \rightarrow R$ by setting, for $x\in
\d(S)$,
$$g((1-t)x+t\lambda (x)) = (1-t)f(x)+tf(\lambda (x)).$$
{\em Claim 2 \/}: $g$ is continuous. To see this, define
$$ \alpha : [0,1]\times \d(S) \to R^n, \quad \alpha(t,x) = (1 - t)x + t\lambda(x).$$
Then $\alpha$ is definable and continuous, $\alpha([0,1]\times \d(S)) = S$, and
$g\circ \alpha$ is continuous,
so $g$ is continuous by p. 96, Corollary 1.13 in \cite{book}.

\medskip\noindent
It is easy to check that if the points $\st(a_0),\dots, \st(a_k)$ in $\bk^n$
are affinely
independent, then $g$ induces a function on $\st(S)$, so in what follows
we assume that $\st(a_0),\dots, \st(a_k)$ are not affinely
independent. Then $\st(S)$ has dimension $d<k$, and we can assume that
$\st(a_0),\dots, \st(a_d)$ are affinely independent.

\smallskip\noindent
{\em Claim 3 \/}: $\st S=\st \d(S)$. To see this, note first that the
affine span of $\st S=[\st(a_0),\dots, \st(a_k)]$ in $\bk^n$ has dimension $d$. Then by a lemma of Carath\'{e}odory
(p. 126 in \cite{book}), each element of $\st S$ is in the convex hull of a
subset of $\{\st(a_0),\dots, \st(a_k)\}$ of size $\le d+1$, and so in
$\st(F)$ for some proper face $F$ of
$S$. This proves Claim 3.

\medskip\noindent
The functions $\lambda$ and $g$ depend on $u$, and without further
specifying $u$
we cannot expect $g$ to
induce a function on $\st S$. We now restrict $u$ as follows:
$a_k\notin H_k=\<a_0,\dots, a_{k-1}\>$ but
$\st(a_k)\in \st H_k$, so we can take $u$ as above such that $a_k+u\in H_k$
and $\st(u)$ is the zero vector of $\bk^n$.

\medskip\noindent
{\em Claim 4 \/}: $d(x,\lambda(x))\in \mathfrak{m}$ for all $x\in \d(S)$.
To see this, note that
$S$ lies between $H_k$ and $H_k+ u$, that is,
$S\subseteq \{x+tu:\ x\in H_k,\ 0\le t \le 1\}$.
This is because $\{x+tu:\ x\in L,\ 0\le t \le 1\}$ is convex, and contains $a_0,\dots, a_k$. For $a\in H_k$,
$$(a+Ru)\cap \{x+tu: x\in H_k,\  0\le t \le 1\} = [a,a+u].$$
Given $x\in \d(S)$ the line $x+Ru$ equals
$a+Ru$ where $(x+Ru)\cap H_k=\{a\}$, and so
$[x, \lambda(x)] \subseteq [a, a+u]$,
so $d(x,\lambda(x)) \le d(a, a+u)$, whence the claim.

It is clear from Claims 2 and 4 that $g$ induces a function on $\st S$.
\end{proof}

\noindent
A {\em subcomplex\/} of a complex $K$ in $R^n$ is a subset $L$ of $K$ such that if
$F$ is a face of any $S\in L$, then $F\in L$; note that then $L$ is also a complex in $R^n$.

\begin{lemma}\label{ext2}
Let $L$ be a subcomplex of a $V$-bounded complex $K$ in $R^n$, and let
$f\colon |L|\to R$ be a $V$-bounded continuous definable
function inducing a function $ \st |L|\to
\boldsymbol{k}$. Then $f$ has a $V$-bounded continuous definable extension
$|K|\to R$ inducing a function $\st |K| \to
\boldsymbol{k}$.
\end{lemma}
\begin{proof} We can assume $L\not= K$, and it
suffices to obtain a strictly
larger subcomplex $L'$ of $K$ and a $V$-bounded continuous definable extension
$f' \colon |L'|\to R$ of $f$ inducing a function
$\st|L'|\to \boldsymbol{k}$.
Take a simplex
$S \in K\setminus L$ of minimal dimension.

Suppose $S = \{ a \}$ with $a \in R^n$. Then $L'=L\cup \{ S \}$ is
subcomplex of $K$ and $L \ne L'$. If
$d(a,|L|)>\ma$, then $f'(a)=0$ determines an extension
of $f$ to $|L'|\to R$ with the required properties. If
$d(a,|L|)\in \ma$, then we can pick $b \in |L|$ such that
$d(a,b)\in \ma$ and define an extension as desired by
$f'(a)=f(b)$.

Next, assume that $S$ is a $k$-simplex with $k>0$. Then all
proper faces of $S$ are in $L$, so $\d(S) \subseteq
|L|$, and by the previous lemma, the function $f|\d(S)$
extends to a $V$-bounded continuous definable function $g\colon
S\to R$ inducing a
function $\st (S )\to \boldsymbol{k}$. Also, $L'= L\cup
\{S \}$ is a subcomplex of $K$, $L\ne L'$, $f$
extends to a $V$-bounded continuous function $f'\colon |L'|\to R$ defined
by $f'(x)=f(x)$ when $x\in |L|$ and $f' (x)=g(x)$ when
$x\in S$, and $f'$ induces a
function $\st (|L'|)\to \boldsymbol{k}$.
\end{proof}

\medskip\noindent
{\bf Good directions.}
In o-minimal triangulation we use extension lemmas in combination with the existence of good directions. For $V$-triangulation we need to sharpen this a little bit. Let
$$\Sp^n:= \{(x_1,\dots, x_{n+1})\in R^{n+1}:\ x_1^2 + \dots + x_{n+1}^2=1\}$$
and define $\Sp^n(\bk)\subseteq \bk^{n+1}$ likewise, with $\bk$ instead of $R$.
A unit vector $u\in \Sp^n$ is a {\em good
direction\/} for a set $X\subseteq R^{n+1}$ if for each $a\in
R^{n+1}$ the line $a+Ru$ intersects $X$ in only finitely many
points. Likewise we define what it means for a vector in $\Sp^n(\bk)$
to be a good direction for a set $X\subseteq \bk^{n+1}$.

\begin{lemma}\label{gooddir}
Let $X\subseteq R^{n+1}$
be definable with $\dim X \le n$. Then there is a good direction $v\in
\Sp^{n}(\boldsymbol{k})$ for $\st X$ such that all $u\in
\Sp^{n}$ with $\st(u)=v$ are good directions for $X$.
\end{lemma}
\begin{proof} We have $\dim{(\st X)} \le n$, for example by Corollary 2.8 in
~\cite{st1}.
Call $u\in
\Sp^{n}$ a {\em bad direction\/} for $X$ if $u$ is not a good
direction for $X$, and define bad directions for $\st X$ similarly.
Let $B\subseteq \Sp^{n}$ be the set of bad directions for $X$, so $B$
is definable and $\dim B < n$ by the Good Directions Lemma on p. 117 of
\cite{book}. Put
$$B':= \st(B) \cup \text{ set of bad directions for }\st X\ \subseteq \Sp^{n}(\boldsymbol{k}).$$
Since $\boldsymbol{k}_{\ind}$ is o-minimal, the set $B'$ is definable in
$\boldsymbol{k}_{\ind}$, and $\dim B' < n$.
It follows that we have a box $C\subseteq \boldsymbol{k}^{n+1}$ such that $C\cap
\Sp^{n}(\boldsymbol{k}) \not=\emptyset$ and
$\closure{C}\cap B' = \emptyset$. Then any $v\in C\cap
\Sp^{n}(\boldsymbol{k})$ has the desired property.
\end{proof}

\noindent
We define a {\em $V$-good direction\/} for a set $X\subseteq R^{n+1}$ to be unit vector
$u\in \Sp^n$ such that $u$ is a good direction for $X$ and $\st(u)\in \Sp^n(\bk)$
is a good direction for $\st X$. The above lemma yields an abundance of $V$-good
directions for $X$ if $X\subseteq R^{n+1}$ is definable with $\dim X \le n$.

\end{section}

\begin{section}{The Triangulation Lemma}

\noindent
In this section we construct a $V$-triangulation of a definable closed $V$-
bounded set in $R^{n+1}$ if a suitable $V$-triangulation of
its projection in $R^n$ is given. First some more notation and terminology.

\medskip\noindent
Let $K$ be a complex in $R^n$. Let $\text{Vert}(K)$ denote the set
of vertices of the simplexes in $K$. Let $(\phi,K)$ be a
triangulation of a definable closed $X\subseteq R^n$, and let $p=p_n^{n+1}
\colon R^{n+1} \to R^n$ be the projection map given by
$$p(x_1,\dots, x_{n+1})=(x_1,\dots, x_n).$$
Then, given definable closed $Y\subseteq R^{n+1}$, a triangulation
$(\psi,L)$ of $Y$ is said to be a {\em lifting\/} of $(\phi,K)$ if
$K=\{p(T): T\in L\}$ (so $K^{\o}  =\{p(T^{\o}  ):\  T\in L\}$)
and the diagram

$$\xymatrix@R+1em@C+7em{
Y \ar[d]\ar[r]^{\psi}      & |L| \ar[d]  \\
X    \ar[r]^{\phi}& |K|
}$$

\medskip\noindent
commutes where the vertical maps are restrictions of $p$ (so
$p(Y)=X$).

To construct liftings we use triangulated sets and
multifunctions
on them, and we proceed to define these notions.
We set
$$\phi^{-1}(K):= \{\phi^{-1}(S):\ S\in K\},$$ and call the pair
$(X, \phi^{-1}(K))$ a {\em triangulated set}.
To simplify notation, let $\cP:=\phi^{-1}(K)$. For $P,Q\in \cP$ we call
$Q$ a {\em face\/} of $P$ if $Q\subseteq P$
(equivalently, $\phi(Q)$ is a face of the simplex $\phi(P)$).
For $P\in \cP$, a {\em proper face\/} of $P$ is a face $Q\in \cP$ of $P$ such that
$P\ne Q$. For $P\in \cP$ we put
$$P^{\o}\ :=\  P \setminus \text{union of the proper faces of }P,$$
so $\phi(P^{\o})=\phi(P)^{\o}$. A point $x\in X$ such that $\{x\}\in
\cP$ (that is, $\phi(x)\in \text{Vert}(K)$) is said to be a {\em
vertex\/} of $(X, \cP)$. A {\em multifunction\/} on $(X, \cP)$ is a
finite collection $F$ of continuous definable functions $f\colon X
\to R$ such that for all $f,g\in F$ and $P\in \cP$, either $f(x) <
g(x)$ for all $x\in P^{\o}$, or $f(x) =g(x)$ for all $x\in P^{\o}$,
or $g(x) < f(x)$ for all $x\in P^{\o}$.

Let $F$ be a multifunction on $(X, \cP)$. For $P\in \cP$ and $f,g\in F$
we say that $g$ is the {\em successor\/} of $f$ on $P$ (in $F$) if
$f(x) < g(x)$ for all $x\in P^{\o}$ (so $f(x) \le g(x)$ for all $x\in P$),
and
there is no $h\in F$ such that $f(x) < h(x) < g(x)$ for all
(equivalently, for some) $x\in P^{\o}$. We set
\begin{enumerate}
\item[(a)] $\Gamma(F):= \bigcup_{f\in F} \Gamma(f) \subseteq R^{n+1}$;
\item[(b)] $F|P:= \{f|P:\  f\in F\}$ for $P\in \cP$;
\item[(c)]  $\cP^F$ is the collection of all sets $\Gamma(f|P)$ with $f\in F$
and $P\in \cP$, and all sets $[f|P, g|P]$ with $P\in \cP$, $f,g\in F$ and
$g$ the successor of $f$ on $P$;
\item[(d)] $X^F: = \text{the union of the sets in }\cP^F$, so
$X^F\subseteq R^{n+1}$.
\end{enumerate}

\noindent
So $\Gamma(F)$ and $X^F$ are closed and bounded in $R^{n+1}$.

\medskip\noindent
The above material in this section does not require the presence of $V$, and
so makes sense and
goes through for any o-minimal field instead of $R$, in particular, for
$\bk_{\ind}$ if the latter is o-minimal. We now bring in $V$ again, and note
that if $(\phi, K)$ is a $V$-triangulation of the definable closed
$V$-bounded $X\subseteq R^n$, then the triangulation
$(\phi_{\st}, \st(K))$ of $\st X$ yields the triangulated set
$(\st X, \st \cP)$, with $\cP:= \phi^{-1}(K)$, and
$$\st \cP:= \{\st P:\ P\in \cP\}=\phi_{\st}^{-1}\big(\st(K)\big).$$

\medskip\noindent
{\bf Remark.} Suppose $\bk_{\ind}$ is o-minimal. Let $(\phi, K)$ be
a $V$-triangulation of the definable closed $V$-bounded set
$X\subseteq R^n$. Let $F$ be a multifunction on $(X, \cP)$ with
$\cP:= \phi^{-1}(K)$, such that each $f\in F$ induces a function
$f_{\st} \colon \st X \to \bk$, and for all $f,g\in F$ the set
$\{y\in \st X:\ f_{\st}(y) =g_{\st}(y)\}$ is a union of sets in $\st
\cP$.
Then $F_{\st}:= \{f_{\st}:\ f\in F\}$ is a multifunction on
$(\st X, \st \cP)$, with
$$\Gamma(F_{\st})=\st(\Gamma F),\quad (\st \cP)^{F_{\st}}=\{\st Q:\ Q\in \cP^F\},
\quad (\st X)^{F_{\st}}= \st(X^F).$$
(The middle equality requires a little thought.)

\begin{lemma}\label{trainglemma} Suppose $\bk_{\ind}$ is o-minimal, and
$\phi, K, X, \cP, F$ are as in the remark above. Assume also that
for all $P \in \cP$ and all $Q \in \st \cP$:
\begin{itemize}
\item[$(\ast)$] if $f , g \in F|P$, $f \not= g$, then
$f(a)\ne g(a)$ for some vertex $a$ of $P$;
\item[$(\ast \ast)$] if $f , g \in F_{\st}|Q$, $f \not= g$,
then $f(a)\ne g(a)$ for some vertex $a$ of $Q$.
\end{itemize}
Then there is a $V$-triangulation $(\psi ,L)$ of $X^{F}$ such
that $(\psi , L)$ is a lifting of $(\phi ,K)$ compatible
with the sets in $\cP^F$, and $(\psi_{\st}, \st(L))$ is a lifting
of $(\phi_{\st} , \st(K))$ compatible with the sets in
$\st(\cP)^{F_{\st}}$.
\end{lemma}
\begin{proof} Choose a total ordering $\le$ on
$\text{Vert} (K)$ such that for all
$a ,b,c \in \text{Vert}(K)$ with $a \le b \le c$ and $\st (a)=\st (c)$ we have
$\st(a)=\st(b)$. This gives a total ordering $\le$ on
$\text{Vert}(\st(K))$ such that if
$a ,b\in \text{Vert}(K)$ and $a \le b$, then $\st (a)\le\st (b)$.
Now $(\phi , K)$, $X$, $F$ are as in the proof of Lemma 2.8, p.129
in \cite{book}, and we
apply the construction from this proof, using the given ordering on
$\text{Vert}(K)$, to obtain a triangulation
$(\psi , L)$ of $X^F$ that
is a lifting of $(\phi , K)$ and is compatible with the sets
in $\cP^F$. We now briefly recall the construction of $(\psi,L)$
from \cite{book}.

 Let $P\in\cP$, let $a_0 , \dots ,a_m$ be the vertices of $\phi (P)$
such that in the ordering above we have $a_0 < a_1 < \dots < a_m$, and let
$f\in F|P$. Then the complex $L(f)$ in $R^{n+1}$ consists
of the $m$-simplex $[b_0 , \dots ,b_m ]$ and all its faces,
where each $b_i = (a_i,r_i)\in R^{n+1}$, $r_i:= f(\phi^{-1} (a_i ))$. Define
$$\psi_f
\colon \Gamma (f) \to |L(f)|, \qquad \psi_f (x, f(x)) := \phi_b (x),$$ where $\phi_b (x)$ is
the point of $[b_0 , \dots ,b_m ]$ with the same affine coordinates
with respect to $b_0 , \dots ,b_m$ as $\phi(x)$ has with respect to $a_0 ,
\dots ,a_m$. Then $\psi_f$ is a homeomorphism.

Suppose in addition that $f$ has a successor $g \in F|P$. Then $L(f,g)$ is the
complex $L$ constructed just before Corollary~\ref{vgood}, so $|L(f,g)|$ is the convex hull
of $\{ b_0 , \dots , b_m , c_0 ,\dots ,c_m   \}$, where
$c_i = (a_i, s_i)\in R^{n+1},\ s_i:=g(\phi^{-1}(a_i ))$. Then the homeomorphism $\psi_{f,g} \colon
[f, g] \to |L(f,g)|$ is given
by $$(x,tf(x) + (1-t)g(x))\mapsto t\phi_b (x)
+(1-t)\phi_c (x),$$ where $\phi_c (x)$ is defined in the same way as
$\phi_b (x)$, and $0\leq t \leq 1$.

The complex $L$ is the union of the complexes $L(f)$ and $L(f,g)$
obtained in this way, and $\psi \colon X^F \to |L|$ extends each of
the $\psi_f$ and $\psi_{f,g}$ above.

\medskip\noindent
Also, $(\phi_{\st} , \st(K))$, $\st X$ and $F_{\st}$
are as in the proof of Lemma 2.8, p. 129  in \cite{book}, with $\bk_{\ind}$
instead of $R$. Thus using the given ordering on
$\text{Vert}(\st(K))$ we construct in the same way as before
a triangulation  $(\theta , M)$ of $(\st
X)^{F_{\st}}$ that is a lifting of $(\phi_{\st} , \st(K))$ and is
compatible with the sets in $\st(\cP)^{F_{\st}}$.

\medskip\noindent
{\em Claim \/}:\  $\psi$ induces $\theta$. To prove this, let $P\in \cP$ and let
$a_0 < \dots < a_m$ be the vertices of the simplex $\phi P\in K$. Put
$Q:=\st P \in \st \cP$, and let the simplex
$\phi_{\st}(Q)=\st(\phi P)\in \st(K)$ have as its vertices
$\alpha_0 <  \dots < \alpha_k $. Then
$$\{0,\dots,m\}=I_0 \cup \dots \cup I_k\
\text{ (disjoint union) with }I_j:=\{i:\ \st(a_i)=\alpha_j\}.$$
Here and later in the proof $i$ ranges over $\{0,\dots,m\}$ and $j$ over
$\{0,\dots,k\}$. Let $f\in F$ and, towards showing that $\psi$ induces
$\theta$ on $\Gamma(f|P)$, put
\begin{align*} b_i &:= (a_i,r_i)\in R^{n+1}, \qquad r_i:= f(\phi^{-1} (a_i )),\\
         \beta_j&:=(\alpha_j, \rho_j)\in \bk^{n+1}, \qquad \rho_j:=f_{\st}(\phi_{\st}^{-1} (\alpha_j )),
\end{align*}
so $\st(b_i)=\beta_j$ for $i\in I_j$. Let  $x\in P$. Then $\phi(x)=\sum_i t_i
a_i$, where all $t_i\geq 0$ and $\sum_{i}t_i = 1$, so
$\phi_{\st}(\st x)=\sum_{j}\tau_j\alpha_j$ with $\tau_j=\sum_{i\in I_j} t_i$.
Then
$$ \psi (x,f(x))=\sum_{i} t_i b_i, \qquad
\theta\big(\st x, f_{\st}(\st x)\big)=\sum_j \tau_j \beta_j,$$ so
$\st\big(\psi(x,f(x))\big)=\theta\big(\st (x, f(x))\big)$. Thus $\psi$ induces
$\theta$ on $\Gamma(f|P)$.

Next, assume also that $f$ has a successor $g\in F$ on $P$, and put
\begin{align*} c_i &:= (a_i,s_i)\in R^{n+1}, \qquad s_i:= g(\phi^{-1} (a_i ))\\
         \gamma_j&:=(\alpha_j, \sigma_j)\in \bk^{n+1}, \qquad \sigma_j:=g_{\st}(\phi_{\st}^{-1} (\alpha_j )),
\end{align*}
so $\st(c_i)=\gamma_j$ for $i\in I_j$. Then, with $\bar{x}:= \st x,\ \bar{t}:= \st t$,
\begin{align*} \psi(x,t f(x) + (1-t)g(x))&=
t\sum_{i}t_i b_i + (1-t)\sum_{i}t_i c_i ,\\
\theta\big(\bar{x}, \bar{t}f_{\st}(\bar{x})+(1-\bar{t})g_{\st}(\bar{x})\big)&=
\bar{t}\sum_j\sigma_j\beta_j  + (1-\bar{t})\sum_j \sigma_j\gamma_j .
\end{align*}
To obtain the second identity, note that either $f_{\st}$ and $g_{\st}$ coincide
on $\st P$, or $g_{\st}$ is the successor of $f_{\st}$ on $\st P$ (in $F_{\st}$).
It follows as with $\Gamma(f|P)$ that $\psi$ induces $\theta$ on $[f|P, g|P]$.
Since $P\in \cP$ was arbitrary, this proves the claim.

\medskip\noindent
For $(\psi,L)$ to have the property stated in the lemma it only remains to
check that $\st(L)=M$. This equality follows from Section~\ref{cc} in view of
how we ordered
$\text{Vert} (K)$ and  $\text{Vert} (\st K)$
and constructed $L$ and $M$.
\end{proof}

\medskip\noindent
{\bf Satisfying conditions $(\ast)$ and $(\ast \ast)$.} In the situation of the remark before the triangulation lemma, condition $(\ast)$ might fail for some $P\in \cP$.
We can then replace $K$ by its barycentric subdivision to satisfy $(\ast)$, as is done in \cite{book}, but a simplex
of this barycentric subdivision is not necessarily a $V$-simplex, so this device
fails to deal with $(\ast \ast)$. Fortunately, a slight generalization of
the barycentric subdivision solves this problem, as we describe below.

\medskip\noindent
Recall that the {\em barycenter\/} of an $m$-simplex $S=[a_0,\dots, a_m]$ in $R^n$ is the point $\frac{1}{m+1}(a_0+ \dots + a_m)$ in $S^0$.
Let $K$ be a complex in $R^n$. A {\em subdivision\/} of $K$ is a complex $K'$ in $R^n$ such that
$|K|=|K'|$ and each simplex of $K$ is a union of simplexes of $K'$; it follows easily
that then each set $S^{\o}$ with $S\in K$ is a union of sets $S'^{\o}$ with $S'\in K'$.
Define a {\em $K$-flag\/} to be a sequence $S_0, \dots, S_k$
in $K$ such that $S_i$ is a proper face of $S_{i+1}$ for all $i<k$. Given such a
$K$-flag and a point $b_i\in S_i^{\o}$ for $i=0,\dots,k$ we have a $k$-simplex
$[b_0,\dots, b_k]$. Assume now that to each $S\in K$ is assigned a point $b(S)\in S^{\o}$.
This yields a subdivision $b(K)$ of $K$ whose simplexes are the
$[b(S_0),\dots, b(S_k)]$ with $S_0,\dots, S_k$ a $K$-flag. (In Chapter 8 of \cite{book}
we took $b(S):= \text{barycenter of }S$, for each $S\in K$, and then $b(K)$ is the
barycentric subdivision of $K$.)

\medskip\noindent
The above paragraph uses only the semilinear structure of $R$, and so
goes through with $\bk$ instead of $R$. We now apply this to a $V$-complex
$K$ in $R^n$ as follows. We choose for each $S\in K$ a point $b(S)\in S^o$ such that
$$ \st\big( b(S)\big)\ =\ \text{barycenter of }\st(S).$$
We claim that then the subdivision $b(K)$ of $K$ has the following property:
$$b(K) \text{ is a $V$-complex, and }\st\big(b(K)\big)\ =\ \text{barycentric subdivision of }\st(K).$$
To see this, let $S_0,\dots, S_k$ be a $K$-flag and $T:=[b(S_0),\dots,b(S_k)]$. Then
$$ \st(T)=[\text{barycenter}\big(\st(S_0)\big),\dots, \text{barycenter}\big(\st(S_k)\big)]$$
is a simplex of the barycentric subdivision of $\st(K)$ (even if the sequence
$\st(S_0), \dots, \st(S_k)$ has repetitions), and each simplex of the barycentric subdivision of $\st(K)$ arises in this way
from a $K$-flag.

\begin{lemma}\label{star}
Assume $\bk_{\ind}$ is o-minimal. Let $K$ be a $V$-complex in $R^n$, $X:= |K|$, and $F$ a multifunction on
$(X,K)$ such that each $f\in F$ induces a function $f_{\st}\colon \st|K| \to \bk$ and for all $f,g\in F$ the set
$$\{y\in X:\; f_{\st}(y)=g_{\st}(y) \}$$ is a union of sets in $\st K$. Then there
is a subdivision $K'$ of $K$ such that $K'$ is a $V$-complex, and
$F$ as a multifunction on $(X,K')$ satisfies the following
conditions for all $P \in K'$ and all $Q \in \st( K')$:
\begin{itemize}
\item[$(\ast)$] if $f , g \in F|P$, $f \not= g$, then
$f(a)\ne g(a)$ for some vertex $a$ of $P$;
\item[$(\ast \ast)$] if $f , g \in F_{\st}|Q$, $f \not= g$,
then $f(a)\ne g(a)$ for some vertex $a$ of $Q$.
\end{itemize}
\end{lemma}
\begin{proof} Just take as $K'$ a complex $b(K)$ as constructed in the
paragraph just before the statement of the lemma. Then $K'$ has the desired properties.
\end{proof}

\medskip\noindent
{\bf Small Paths.}
To apply the triangulation lemma in the next section we also
need to construct a multifunction. This will require the extension
lemma~\ref{ext2} as well as the lemma below about the ``small path'' property.
In the rest of this section $\bk_{\ind}$ is o-minimal, and we consider
a definable $V$-bounded set $X\subseteq R^n$. We say that $X$ has {\em small paths\/} if for all
$x,y\in X$ with $\st x = \st y$ there is $\epsilon\in \ma^{>0}$ and
a definable continuous path $\gamma \colon [0,\epsilon] \to X$ such
that $\gamma(0)=x,\ \gamma(\epsilon)=y$, and $\st(\gamma(t))=\st x$
for all $t\in [0,\epsilon]$; such $\gamma$ will be called a {\em small path}.
Note that if $X$ is convex, then
$X$ has small paths. It follows that if
there is a $V$-bounded simplex $S$ in $R^n$ and a
definable homeomorphism $X \to S$ inducing a
homeomorphism $\st X \to \st S$, then $X$ has small paths.

\begin{lemma}\label{con6} Assume $X$ has small paths, and let
$f\colon X \to R$ be definable, continuous, and $V$-bounded, such
that the upward unit vector $e_{n+1}\in \bk^{n+1}$ is a good
direction for $\st(\Gamma f)$. Then
$f$ induces a function $ \st X \to \bk$.
\end{lemma}
\begin{proof} Let $x,y\in X$ be such that $\st x=\st y$; it is enough to show that
then $\st f(x)= \st f(y)$. Take a
small path $\gamma \colon [0,\epsilon ]\rightarrow X$ such that
$\gamma (0)=x $ and $\gamma (\epsilon )=y$. Then the standard parts of the points
$\big(\gamma(t), f(\gamma(t))\big)$ all lie on the same  vertical line in $\bk^{n+1}$, and since $e_{n+1}$ is a good direction for $\st(\Gamma f)$, this yields
$\st f(x)=\st f(y)$.
\end{proof}

\end{section}

\begin{section}{Proof of $V$-triangulation}

\noindent Recall the $V$-triangulation theorem stated on page
\pageref{triangthm}:

\begin{theorem}\label{tr} Suppose the boolean algebra $\Def^2(\bk_{\ind})$ is generated by
its subset $\{\st(X):\  X\in \Def^2(R)\}$. Then every $V$-bounded closed definable
$X\subseteq R^n$ with definable subsets
$X_1 , \dots , X_k$ has a $V$-triangulation
compatible with $X_1 , \dots, X_k$.
\end{theorem}

\noindent
Before we start the proof, first note that the hypothesis of the theorem
implies that $\Def^1(\bk_{\ind})$ is generated by
$\{\st(X):\  X\in \Def^1(R)\}$, which in turn is equivalent to
$\bk_{\ind}$ being o-minimal. If $\bk_{\ind}$ is o-minimal,
the conclusion of the theorem clearly holds for $n=1$. The proof will
show that the conclusion of the theorem for $n=2$ also follows just from
assuming that $\bk_{\ind}$ is o-minimal. The stronger hypothesis about
$\Def^2(\bk_{\ind})$ will only be used to obtain the conclusion of the theorem
for $n>2$.

\begin{proof} As already noted, $\bk_{\ind}$ is o-minimal, and
the theorem holds for $n=1$. We proceed by induction on $n$, so assume
inductively that for a certain $n\ge 1$: \begin{enumerate}
\item[(i)] $\Def^n(\bk_{\ind})$ is generated by $\{\st(X):\  X\in \Def^n(R)\}$;
\item[(ii)] every $V$-bounded closed definable $X\subseteq R^n$ with definable
subsets $X_1 , \dots , X_k$ has a $V$-triangulation
compatible with $X_1 , \dots, X_k$.
\end{enumerate}
{\em Claim.} $\C(n)$ holds. To prove this claim, let $Z\in \Def^n(\bk_{\ind})$
be closed and bounded in $\bk^n$; we have to show that $Z=\st(Q)$ for some
$Q\in \Def^n(R)$. Now by part (i) of the inductive assumption,
$Z$ is a boolean combination of
sets $\st(X_1),\dots, \st(X_k)$ with $X_1,\dots, X_k\in \Def^n(R)$, and we
can assume that $X_1,\dots, X_k$ are $V$-bounded.
Take a $V$-bounded closed $X\in \Def^n(R)$ containing all $X_i$ as subsets
and such that $Z\subseteq \st(X)$. Then by part (ii) of the inductive assumption
we have triangulated sets $(X, \cP)$ and $(\st X, \st \cP)$
such that each $X_i$ is a union of sets $P^{\o}$ with $P\in \cP$. Then
each $\st(X_i)$ is a union of sets $\st(P^{\o})=\st(P)\in \st(\cP)$. Each
$\st(P)$ with $P\in \cP$ is a union of sets from the partition
$\st(\cP)^{\o}=\{\st(P)^{\o}:\ P\in \cP\}$ of $\st(X)$, so each
$\st(X_i)$ is such a union as well, and so is their boolean combination $Z$.
But $Z$ is closed, so $Z$ is then a union of closures $\st(P)$ of sets
$\st(P)^{\o}$ with $P\in \cP$, and so $Z=\st(Q)$ where $Q$ is a union of
sets $P\in \cP$. This proves the claim.

Then by Lemma~\ref{closed then sigma} we have
$(R,V)\models \Sigma (n)$. Also (i) holds
with $n+1$ instead of $n$: for $n=1$ this is
just the hypothesis of the theorem, and if $n\ge 2$, then $\Sigma$ holds
by the claim and Section 2, and so we can use \cite{st1}.

In order to prove that (ii) holds with $n+1$ instead of $n$, let $Y\subseteq R^{n+1}$ be $V$-bounded, closed, and definable,
with definable subsets $Y_1,\dots, Y_k$; our aim is then to construct a
$V$-triangulation of $Y$ compatible with $Y_1,\dots, Y_k$. Put
$$T:= \text{bd}(Y_0) \cup \text{bd}(Y_1 )\cup \dots \cup \text{bd}(Y_k
), \qquad Y_0:= Y.$$ Then $\dim T \le n$, so by Lemma \ref{gooddir} and an argument
as in the proof of 2.9, p.130 in \cite{book}, we can assume
that $e_{n+1}$ is a $V$-good direction for $T$.

We are going to construct a $V$-triangulation of $X:= p^{n+1}_{n}T =
p^{n+1}_{n}Y\subseteq R^n$ so that we can use the triangulation lemma~\ref{trainglemma}.

Cell decomposition gives a finite
partition $\mathcal{C}$ of $X$ into cells $C$ such that $T \cap (C \times R)$
is the union of the graphs of definable continuous functions
$$f_{C,1}< \dots < f_{C,l(C)} : C \to R, \quad l(C)\ge 1,$$
such that for $i=0,\dots,k$ and $j=1,\dots,l(C)$,
\begin{align*} \text{either }&\Gamma(f_{C,j})\subseteq Y_i\ \text{ or }\ \Gamma(f_{C,j})\cap Y_i= \emptyset, \text{ and for }1\le j<l(C):\\
\text{either }&(f_{C,j}, f_{C,j+1})\subseteq Y_i\ \text{ or }\  (f_{C,j}, f_{C,j+1})\cap Y_i= \emptyset.
\end{align*}
Since $e_{n+1}$ is a good direction for $T$ and $T\supseteq
\cl(\Gamma f)$ for each $f=f_{C,j}$, each $f_{C,j}$ extends
continuously to a definable function $\cl(C)\to R$, and we denote
this extension also by $f_{C,j}$. We need to extend these functions
$f_{C,j}$ to all of $X$ in a nice way, and towards this goal we note
that the inductive assumption (ii) gives a $V$-triangulation $(\phi
, K)$ of $X$ compatible with all $C\in \mathcal{C}$. This gives a
triangulated set $(X, \cP)$ with $\cP:= \phi^{-1}(K)$. Let $C\in \C$
be given. Then the set $\cl(C)$ is a finite union of sets $P\in
\cP$. The sets $P\in \cP$ have small paths, and so by
Lemma~\ref{con6} each function $f_{C,j} \colon \cl(C) \to R$ induces
a function on $\st( \cl(C))$, and thus, by Lemma~\ref{ext2}, extends
to a definable continuous $V$-bounded function $f\colon X \to R$
such that $f$ induces a function $\st(X) \to \bk$. In this way we
obtain a finite set $F$ of definable continuous $V$-bounded
functions $f\colon X \to R$ such that each $f\in F$ induces a
function $f_{\st} \colon \st X \to \bk$, each $f\in F$ extends some
$f_{C,j}$, and each $f_{C,j}$ has an extension in $F$. To make $F$
into a multifunction on $(X, \cP)$ that induces a multifunction on
$(\st X, \st \cP)$ we may have to refine $\cP$, and this is done as
follows. Since $\Sigma(n)$ holds in $(R,V)$, we have $\epsilon_0 \in
\ma^{>0}$ such that for all $f,g\in F$ and $\epsilon\in
\ma^{>\epsilon_0}$,
$$
\st \{x\in X :\ |f(x)-g(x)|\leq \epsilon_0 \}\  =\
\st \{x\in X:\  |f(x)-g(x)|\leq \epsilon\},  $$
and thus for all $f,g\in F$,
$$ \st \{x\in X :\ |f(x)-g(x)|\leq \epsilon_0 \}\  =\
\st \{x\in X:\  f(x)-g(x)\in \ma\}. $$
Using again the inductive assumption (ii) we arrange that our
$V$-triangulation $(\phi , K)$ above is also
compatible with all sets $$\{x\in X:\ f(x)=g(x)\}\ \text{ and }\ \{x\in X:\  |f(x)-g(x)|\leq \epsilon_0\}, \qquad (f,g\in F).$$
Note that then $F$ is a multifunction on $(X, \cP)$, and that for all $Y_i$
and $f\in F$ and $P\in \cP$, either $\Gamma(f|P^o)\subseteq Y_i$ or $\Gamma(f|P^o)\cap Y_i= \emptyset$, and if also $g\in F$ is the successor of $f$ on $P$,
then either $(f|P^o, g|P^o)\subseteq Y_i$ or
$(f|P^o, g|P^o)\cap Y_i= \emptyset$. Note that for all $f,g\in F$,
$$\st \{x\in X:\  |f(x)-g(x)|\leq \epsilon_0\}\ =\ \{y\in \st X:\ f_{\st}(y)=g_{\st}(y)\},$$
and the set on the left is a union of sets in $\st \cP$. Hence
we are in the situation of the remark preceding
Lemma~\ref{trainglemma}, so
$F_{\st}:= \{f_{\st}:\ f\in F\}$ is a multifunction on $(\st X, \st \cP)$.
By Lemma \ref{star} we can replace $K$ by a subdivision and $\cP$ accordingly to arrange also that for all $P\in \cP$ and $Q\in \st(\cP)$ conditions $(\ast)$ and $(\ast \ast)$ of Lemma
\ref{trainglemma} are satisfied. This triangulation lemma then yields a $V$-triangulation $(\psi, L)$ of $X^F$ that lifts $(\phi,K)$ and is compatible with the sets in $\cP^F$, and such that $(\psi_{\st}, \st(L))$ is a lifting of
$(\phi_{\st}, \st(K))$ compatible with the sets in $\st(\cP)^{F_{\st}}$.
Let $L'$ be the subcomplex of $L$ for which $|L'|=\psi(Y)$, and put $\psi':= \psi|Y$.
Then $(\psi', L')$ is a $V$-triangulation of $Y$ compatible with $Y_1,\dots, Y_k$, as promised.
\end{proof}

\noindent
In the course of the proof just given we have also established the
implication $(6) \Rightarrow (1)$ of Theorem~\ref{equi}, and this concludes
the proof of that theorem, by remarks following its statement.

\end{section}

\begin{section}{Two applications of $V$-triangulation}

\noindent
In this section we assume that $(R,V)$ satisfies the (equivalent) conditions of
Theorem~\ref{equi}.
Here is an easy consequence of $V$-triangulation and Lemma~\ref{ext2}:

\begin{corollary} Let $X,Y\subseteq R^n$ be closed and $V$-bounded definable sets with $X\subseteq Y$,
and let $f: X \to R$ be a continuous $V$-bounded definable function inducing a function $\st X \to \bk$. Then $f$ extends to a continuous $V$-bounded definable function inducing a function
$\st Y \to \bk$.
\end{corollary}

\noindent
Does this go through if the assumption that $X$ and $Y$ are closed is replaced by the weaker one that $X$ is closed in $Y$? That would give a $V$-version of the o-minimal Tietze extension result (3.10) of Chapter 8 in \cite{book}.

\medskip\noindent
{\bf A Finiteness Result.} Let $X\subseteq
R^m$ and $Y\subseteq R^n$ be $V$-bounded and definable. Then a $V$-homeomorphism
$f \colon X \to Y$ is by definition a definable homeomorphism $X \to
Y$ that induces a homeomorphism $f_{\st}\colon \st X \to \st Y$.

\medskip\noindent
For a $V$-bounded definable $X\subseteq R^{m+n}$
the sets $X(a)\subseteq R^n$ with $a\in R^m$ fall into only finitely many  $V$-homeomorphism types. Towards proving this (in a stronger form), consider
triples $(N, \C, E)$ such that $N\in \mathbb{N}$, $\C$ is a collection of nonempty subsets of $\{1,\dots,N\}$ with $\{i\}\in \C$ for $i=1,\dots,N$ and
$I\in \C$ whenever $I$ is a nonempty subset of some $J\in \C$, and $E$ is an equivalence relation on $\{1,\dots,N\}$. Note that for any given $N\in \mathbb{N}$
there are only finitely many such triples $(N, \C, E)$, so in total there are
only countably many such triples.

Let $(N, \C,E)$ be a triple as above. We say that a $V$-complex $K$ in $R^n$
is of type $(N, \C, E)$ if there is a bijection $i: \text{Vert}(K) \to \{1,\dots,N\}$
such that $\C$ is the collection of sets
$\{i(a):\ a \text{ is a vertex of }S\}$ with $S\in K$,
and for all $a,b\in \text{Vert}(K)$, $i(a)E\ i(b)\ \Leftrightarrow\  \st(a)=\st(b)$.
Suppose the $V$-complexes $K$ in $R^n$  and $K'$ in $R^{n'}$ are both of type
$(N, \C, E)$, witnessed by the bijections $i: \text{Vert}{K} \to \{1,\dots, N\}$ and
$j: \text{Vert}(K') \to \{1,\dots,N\}$. We claim that then $|K|$ and $|K'|$ are
$V$-homeomorphic. To see this, note that the map
$v:=j^{-1}\circ i: \text{Vert}(K) \to \text{Vert}(K')$ is a bijection such that
\begin{enumerate}
\item[(i)] for all $a_0,\dots, a_k\in \text{Vert}(K)$, $a_0,\dots, a_k$ are the
vertices of a simplex in $K$ iff
$va_0,\dots, va_k$ are the vertices of a simplex in $K'$;
\item[(ii)] for all $a,b\in \text{Vert}(K)$, $\st(a)=\st(b)$ iff
$\st(va)=\st(vb)$.
\end{enumerate}

\noindent By (i) we can extend $v$ uniquely to a homeomorphism $\phi
\colon |K| \to |K'|$ that is affine on each simplex of $K$. Using Lemma~\ref{vg} and the assumption that $K$ and $K'$ are $V$-complexes
it then follows from (ii) that $\phi$ is a $V$-homeomorphism.

For the proof below it is convenient to fix a sequence of
$V$-complexes $K_1, K_2, K_3,\dots$ in $R^n$ such that every $V$-complex $K$ in $R^n$
is of the same type $(N, \C, E)$ as some complex in this sequence.

\begin{corollary} Let $Z\subseteq R^m$ be definable, and let
$X\subseteq Z\times R^n\subseteq R^{m+n}$ be definable such that each section $X(a)$
with $a\in Z$ is $V$-bounded. Then there is a partition of $Z$ into subsets $Z_1,\dots, Z_k$, definable in $(R,V)$,
such that if $a,b\in Z$ are in the same $Z_i$, then $X(a)$ and $X(b)$ are $V$-homeomorphic.
\end{corollary}
\begin{proof} We shall establish this in the stronger form that
there are $M\in \N$ and definable sets
$\Phi_1,\dots, \Phi_l\subseteq R^M\times R^{2n}$ such that
for each $a\in Z$ there is $j\in \{1,\dots,l\}$ and $b\in R^M$ for which $\Phi_j(b)\subseteq R^{2n}$ is the graph of a
map $\phi: \cl(X(a)) \to |K_j|$ that makes $(\phi,K_j)$ a $V$-triangulation of
$\cl(X(a))$ compatible with $X(a)$. For simplicity, assume that
$X(a)$ is closed for all $a\in Z$; the general case is very similar.
To prove the stronger statement we can assume that $(R,V)$ is $\kappa$-saturated
with uncountable $\kappa > |L|$ where $L$ is the language of $\text{Th}(R)$.
Consider $L$-formulas $\phi(u,x,y)$ where $u=(u_1,u_2,\dots)$ is an infinite
sequence of variables and $x=(x_1,\dots, x_n)$, $y=(y_1,\dots, y_n)$. (Of course,
in each such $\phi(u,x,y)$ only finitely many $u_i$ occur.) By $V$-triangulation and saturation there are such formulas $\phi_1(u,x,y), \dots, \phi_l(u,x,y)$ such that
for each $a\in Z$ there is $j\in \{1,\dots,l\}$ and $b\in R^{\N}$ for which
$\phi(b,x,y)$ defines the graph of a map $\phi: X(a) \to |K_j|$ that makes
$(\phi, K_j)$ is a $V$-triangulation of $X(a)$. Now take $M\in \N$ such that
no variable $u_i$ with $i>M$ occurs in any of the $\phi_j$. With this $M$ the
claim at the beginning of the proof is established.
\end{proof}

\end{section}

\end{document}